\documentclass{article}
\usepackage{multicol,graphicx,color}
\usepackage{pslatex}
\usepackage{authblk}
\usepackage{amsthm}
\usepackage{amsmath}
\usepackage{amssymb}
\usepackage{color}
\usepackage{latexsym}
\usepackage{lscape}
\usepackage{epsfig}
\usepackage{pstricks}
\usepackage{amsfonts}
\usepackage{mathrsfs}
\usepackage{mathrsfs}
\usepackage{enumitem}
\usepackage[numbers,sort&compress]{natbib}
\usepackage[
  hmarginratio={1:1},     % equal left and right margins
  vmarginratio={1:1},     % equal top and bottom margins=
  textwidth=15cm,        % new text width
  textheight=21cm,
  heightrounded,          % always useful
]{geometry}

\usepackage{graphicx,color}
\usepackage[colorlinks]{hyperref}
\hypersetup{linkcolor=blue,citecolor=blue,filecolor=black,urlcolor=blue}

\theoremstyle{plain}
\newtheorem{theorem}{Theorem}

\newtheorem{lemma}[theorem]{Lemma}
\newtheorem{proposition}[theorem]{Proposition}
\theoremstyle{definition}
\newtheorem{definition}[theorem]{Definition}
\newtheorem{remark}[theorem]{Remark}

\numberwithin{equation}{section}
\numberwithin{theorem}{section}

\author{Yuxin Li\footnote{liyx097@nenu.edu.cn}} \affil{School of Mathematics and Statistics,
	Northeast Normal University, Changchun 130024, Jilin,
	PR China}

\author{Xiaojun Chang \footnote{changxj100@nenu.edu.cn}} \affil{School of Mathematics and Statistics \& Center for Mathematics and Interdisciplinary Sciences,
 Northeast Normal University, Changchun 130024, Jilin,
PR China}

\author{Zhaosheng Feng \footnote{zhaosheng.feng@utrgv.edu}} \affil{School of Mathematical and Statistical Sciences, University of Texas Rio Grande Valley, Edinburg, TX 78539, USA}

\title{Normalized solutions for Sobolev critical Schr\"odinger-Bopp-Podolsky systems}
\date{}

\begin{document}

\maketitle

\begin{abstract}
	We study the Sobolev critical Schr\"odinger-Bopp-Podolsky system
\begin{gather*}
 -\Delta u+\phi u=\lambda u+\mu|u|^{p-2}u+|u|^4u\quad \text{in }\mathbb{R}^3,\\
 -\Delta\phi+\Delta^2\phi=4\pi u^2\quad \text{in } \mathbb{R}^3,
\end{gather*}
under the mass constraint
\[
 \int_{\mathbb{R}^3}u^2\,dx=c
\]
for some prescribed $c>0$, where $2<p<8/3$, $\mu>0$ is a parameter,
and $\lambda\in\mathbb{R}$ is a Lagrange multiplier. By developing a constraint
minimizing approach, we show that the above system admits a local minimizer.
Furthermore, we establish the existence of normalized ground state solutions.
\end{abstract}

\medskip


\section{Introduction}\label{intro}\setcounter{equation}{0}

We consider the Schr\"odinger-Bopp-Podolsky system
\begin{equation}\label{SBP}
\begin{gathered}
 -\Delta u+\phi u=\lambda u+\mu|u|^{p-2}u+|u|^4u
 \quad \text{in }\mathbb{R}^3,\\
 -\Delta\phi+\Delta^2\phi=4\pi u^2\quad \text{in }\mathbb{R}^3,
\end{gathered}
\end{equation}
where $u,\phi:\mathbb{R}^3\to\mathbb{R}$, $\mu>0$, $\lambda\in\mathbb{R}$ and
$2<p<8/3$. System \eqref{SBP} was suggested as a model to describe solitary
waves for nonlinear Schr\"odinger equation coupled with an electromagnetic
field in the Bopp-Podolsky electromagnetic theory \cite{Bopp1940,Podolsky1942}.
The functions $u$ and $\phi$ denote the modulus of the wave function and the
electrostatic potential, respectively. The Bopp-Podolsky theory is a
second-order gauge theory of the electromagnetic field, which was developed
by Bopp \cite{Bopp1940} and Podolsky \cite{Podolsky1942} independently to
solve the alleged infinity problem in classical Maxwell theory.
For more physical applications, we refer the reader to
\cite{BPVZ2017,Born-1934,BI-1934, BPS2014, BPS2017, CDMPP2018, Feng2022, KT1948}
and the references therein.

In the recent decades, considerable attention has been given to the
Schr\"odinger-Bopp-Podolsky system  from quite a few scientific fields.
Siciliano-D'Avenia \cite{DS2019} studied a Schr\"odinger-Bopp-Podolsky system
with Sobolev subcritical growth,
\begin{equation}\label{DS2019}
 \begin{gathered}
 -\Delta u+\omega u +q^2\phi u=|u|^{p-2}u\quad\text{in }\mathbb{R}^3,\\
 -\Delta\phi+a^2\Delta^2\phi=4\pi u^2\quad\text{in }\mathbb{R}^3,
 \end{gathered}
\end{equation}
where $a>0$, $\omega>0$, $q\neq 0$, and $p\in(2,3]$.
They obtained the existence and nonexistence results depending on the
various ranges of $p$ and $q$ and showed that, in the radial case,
those solutions tend to the solutions of the classical
Schr\"odinger-Poisson equation as $a\to0$.
Silva-Siciliano \cite{SS2020} proved that the system has no solutions for
large $q$ and has two radial solutions for small $q$.
They also presented qualitative properties about the energy level of the
solutions and a variational characterization of these extremal values of $q$.
Wang-Chen-Liu \cite{WCL2022} established the existence, multiplicity and
asymptotic behavior of solutions for the Schr\"odinger-Bopp-Podolsky system
with general nonlinearities.
Figueiredo-Siciliano \cite{FS2023} proved the existence and multiplicity
of solutions for the Schr\"odinger-Bopp-Podolsky
system under positive potentials.

Chen-Tang \cite{CTcritical2020} studied a critical Schr\"odinger-Bopp-Podolsky
system with a subcritical perturbation,
\begin{equation}\label{CT2020}
 \begin{gathered}
 -\Delta u+V(x)u+\phi u=\mu f(u)+u^5\quad\text{in }\mathbb{R}^3,\\
 -\Delta\phi+a^2\Delta^2\phi=4\pi u^2\quad\text{in }\mathbb{R}^3,
 \end{gathered}
\end{equation}
where $a>0$, $V$, and $f$ are continuous functions, and
$\int_0^tf(s)\,ds\ge t^p$ for $p\in(2,6)$ and $t\ge0$.
They showed that system \eqref{CT2020} admits ground state solutions under
certain conditions of $V$ and $f$. Using different variational techniques,
Li-Pucci-Tang \cite{LPT2020} obtained the existence of a nontrivial ground
state solution for \eqref{CT2020} when $f(u)=|u|^{p-1}u$ and its limit
system in the sense that $V(x)\to V_{\infty}\in \mathbb{R}^+$ as
$|x|\to +\infty$. Subsequently, Hu-Wu-Tang \cite{HWT2023} established the
existence of least energy sign-changing solutions of \eqref{CT2020}.
For more recent results, we refer to
\cite{Bor2023,lin2022,MS2023,PJ2023,Qu2022,Zh2022,ZCC2021}.

Note that the papers mentioned above on system \eqref{DS2019} assume
$\omega\in\mathbb{R}$ as a fixed parameter to study nontrivial solutions.
Alternatively, we can search for solutions with the prescribed $L^2$-norm
for system \eqref{SBP}. This approach seems to be meaningful from the
physical point of view because of the conservation of mass.
In the present study, we focus on finding normalized solutions to \eqref{SBP},
i.e.\ a couple $(u,\lambda)\in H^1(\mathbb{R}^3)\times\mathbb{R}$ satisfies
\eqref{SBP} together with the normalization condition
\begin{equation}\label{massconstraint}
 \int_{\mathbb{R}^3}|u|^2\,dx=c
\end{equation}
for a priori given $c>0$. As we know, for each $u\in H^1(\mathbb{R}^3)$,
there exists a unique solution $\phi=\phi_u$ to the second equation of \eqref{SBP}
satisfying
\[
 \phi_u(x)=\frac{1}{4\pi}\int_{\mathbb{R}^3}\frac{1-e^{-|x-y|}}{|x-y|}u^2(y)\,dy.
\]
Then, system \eqref{SBP} is reduced into an equivalent integro-differential form
\begin{equation}\label{sbp}
 -\Delta u+\phi_u u=\lambda u+\mu|u|^{p-2}u+|u|^4u\quad \text{in }\mathbb{R}^3.
\end{equation}
It is standard that for any $c>0$, a solution of \eqref{sbp}-\eqref{massconstraint}
can be regarded as a critical point of the corresponding Energy functional
\[
 I(u):=\frac{1}{2}\int_{\mathbb{R}^3}|\nabla u|^2\,dx
+\frac{1}{4}\int_{\mathbb{R}^3}\phi_uu^2\,dx
-\frac{\mu}{p}\int_{\mathbb{R}^3}|u|^p\,dx-\frac16\int_{\mathbb{R}^3}|u|^6\,dx,
\]
restricted to
\[
 S(c):=\big\{u\in H^1(\mathbb{R}^3): \int_{\mathbb{R}^3}u^2\,dx=c\big\}.
\]
Then the parameter $\lambda\in\mathbb{R}$ appears as a Lagrange multiplier.
It is easy to verify that $I(u)$ is a well-defined and $C^1$ functional on
$H^1(\mathbb{R}^3)$.  Recently, there are numerous contributions flourishing
within this topic, for instance, see
\cite{BM2021,bartsch2016, bartsch2019, bartsch2021, bhigou,  CJS2023, CJ2019,
Jean1997,JJLV2022,jeanjeanttl,soavenls,S2020-2,weiwu}.

\begin{definition} \rm
 We say that $\tilde{u}\in S(c)$ is a ground state solution of \eqref{sbp}
if it is a solution having minimal energy among all the solutions which belong
to $S(c)$, i.e.,
 \[
 dI|_{S(c)}(\tilde{u})=0, \quad
 I(\tilde{u})=\inf\{I(u): dI|_{S(c)}(u)=0,\ u\in S(c)\}.
 \]
\end{definition}

Afonso-Siciliano \cite{AS2021} considered the existence of normalized solutions
for the Schr\"odinger-Bopp-Podolsky system in bounded domains under Neumann
boundary conditions. He-Li-Chen \cite{HLC2022} investigated the following system,
\begin{equation}\label{HLC2022}
 \begin{gathered}
 -\Delta u+\omega u+\phi u=|u|^{p-2}u\quad\text{in }\mathbb{R}^3,\\
 -\Delta\phi+a^2\Delta^2\phi=4\pi u^2\quad\text{in }\mathbb{R}^3,\\
 \|u\|_{L^2}^2=\rho,
 \end{gathered}
\end{equation}
where $\omega\in\mathbb{R}$, $a>0$, and $\rho>0$. By the minimizing method, they obtained
the existence of normalized solutions for $\omega>0$, $a=1$, and
$p\in(2,\frac{10}{3})$, in which the corresponding functional is bounded from
below on $S(c)$. Ramos-Siciliano \cite{RS2023} proved that if $2<p<3$, $\rho>0$
is sufficiently small or if $3<p<\frac{10}{3}$, $\rho>0$ is sufficiently large,
then system \eqref{HLC2022} admits a least energy solution.
Moreover, in the case of $2<p<\frac{14}{5}$ and $\rho>0$ small enough, the least
energy solutions are radially symmetric up to translation and converge to a least
energy solution of the Schr\"odinger-Poisson-Slater system under the same $L^2$-norm
constraint.

We remark that the above papers do not involve the $L^2$-supcritical case where
$p>\frac{10}{3}$. Since in such a situation, the classical methods for dealing
with $L^2$-supercritical problems fail due to the fact that $\phi_u$ is not
homogeneous, which is difficult for us to make use of the scaling of type
$t\mapsto t^{\alpha}u(t^{\beta}\cdot)$ for $\alpha, \beta\in\mathbb{R}$ and $t>0$.
Moreover, the appearance of the term
$\int_{\mathbb{R}^3}\int_{\mathbb{R}^3}e^{-|x-y|}u^2(x)u^2(y)\,dx\,dy$ in the
corresponding Pohozaev identity is another obstacle to deal with.
It is worthy to note that in the case of $a=0$, this problem reduces to the
well-known Schr\"odinger-Poisson-Slater system
 \begin{equation}\label{SPsystem}
 \begin{gathered}
  -\Delta u+b_1(|x|^{-1}\ast|u|^2)u+\lambda u=b_2|u|^{p-2}u
  \quad\text{in }\mathbb{R}^3,\\
  \|u\|_{L^2}^2=c,
 \end{gathered}
 \end{equation}
where $b_1$, $b_2\in \mathbb{R}$, and $p\in(\frac{10}{3},6]$.
Bellazzini-Jeanjean-Luo \cite{jean2013} proved that if $b_1, b_2>0$,
then \eqref{SPsystem} admits a solution of mountain pass type for $c>0$ sufficiently
small by using a mountain-pass argument. Recently, this result has been developed
by Jeanjean-Le \cite{jeanjeanttl} under different assumptions on $b_1$, $b_2$, and $p$.
Li-Zhang \cite{LZ2023} investigated system \eqref{SBP} with a Sobolev critical term.
For $p\in(\frac{10}{3},6)$, by applying a mountain-pass argument, they established
the existence of positive normalized ground state solutions to \eqref{SBP} for
large $\mu>0$ and small $c>0$. For $p\in(2,\frac{10}{3}]$, by combining the mountain
pass theorem with Lebesgue dominated convergence theorem, they proved the existence
of normalized ground state solutions for large $\mu>0$ and small $c>0$.

Because of the critical term $|u|^4u$, it is not difficult to check that
$I|_{S(c)}$ is unbounded from below. However, as we see, the combined action of
$L^2$ subcritical term $\mu|u|^{p-2}u$ and the nonlocal term $\phi_uu$ creates
a geometry of local minima of $I$ on $S(c)$ for $c>0$ small enough.
Based on \cite{JJLV2022,S2020-2}, there is a natural question whether $I|_{S(c)}$
has a critical point which is a local minimizer in the case where
$p\in(2,\frac{8}{3})$. This constitutes the main motivation of this study and our
goal is to make an effort to find a positive answer to this question.

For $u\in S(c)$, we set
\[
 u^t(x):=t^{3/2}u(tx),\quad t>0,\; x\in\mathbb{R}^3.
\]
A direct calculation leads to
\begin{equation}\label{fibermap}
\begin{aligned}
 \Phi_u(t):=I(u^t)&=\frac{t^2}{2}\int_{\mathbb{R}^3}|\nabla u|^2\,dx
  +\frac{t}{16\pi}\int_{\mathbb{R}^3}\int_{\mathbb{R}^3}
  \frac{1-e^{-\frac1{t}|x-y|}}{|x-y|}u^2(x)u^2(y)\,dx\,dy\\
&\quad-\frac{\mu\, t^{\frac{3(p-2)}{2}}}{p}\int_{\mathbb{R}^3}|u|^p\,dx
 -\frac{t^{6}}{6}\int_{\mathbb{R}^3}|u|^6\,dx,
\end{aligned}
\end{equation}
which is the so-called fiber map and plays an important role in the discussion
of the geometrical  structure of the functional $I$. At this stage,
we introduce the related Pohozaev manifold defined by
\[
 \Lambda(c):=\{u\in S(c): Q(u)=0\},
\]
where
\begin{equation}\label{pohozaevtypeid}
\begin{aligned}
 Q(u)&:=\frac{d}{dt}\big|_{t=1}I(u^t) \\
 &=\int_{\mathbb{R}^3}|\nabla u|^2\,dx
  +\frac{1}{16\pi}\int_{\mathbb{R}^3}\int_{\mathbb{R}^3}
   \frac{1-e^{-|x-y|}}{|x-y|}u^2(x)u^2(y)\,dx\,dy\\
 &\quad -\frac{1}{16\pi}\int_{\mathbb{R}^3}\int_{\mathbb{R}^3}
  e^{-|x-y|}u^2(x)u^2(y)\,dx\,dy
  -\frac{3\mu\,(p-2)}{2p}\int_{\mathbb{R}^3}|u|^p\,dx\\
&\quad -\int_{\mathbb{R}^3}|u|^6\,dx.
\end{aligned}
\end{equation}
As mentioned earlier, for any fixed $\mu>0$, we can find a $c_0=c_0(\mu)>0$
such that, for any $c\in(0,c_0)$ there exists an open set $V(c)\subset S(c)$
with the property
\[
 m(c):=\inf_{u\in V(c)} I(u)<0<\inf_{u\in\partial V(c)}I(u),
\]
where
\[
 V(c):=\{u\in S(c):\|\nabla u\|_2^2<\rho_0\},~~~\partial V(c):=\{u\in S(c):\|\nabla u\|_2^2=\rho_0\}
\]
for a suitable $\rho_0>0$ only depending on $c_0>0$.

In the process of minimization, the key difficulty is the lack of compactness of
the bounded minimizing sequence $\{u_n\}\subset V(c)$ and the most critical step
is to prove the strong subadditivity inequality
\begin{equation}\label{strict add}
 m(c)<m(c_1)+m(c_2)~~\text{for all}~~0<c_1,c_2<c,
\end{equation}
which is a sufficient condition to ensure that any minimizing sequence on $V(c)$
is relatively compact. Moreover, \eqref{strict add}
is a stronger version of the so-called  weak subadditivity inequality
\begin{equation}\label{weak add}
 m(c)\leq m(c_1)+m(c_2)\quad \text{for all }0<c_1,c_2<c.
\end{equation}
However, because of $p\in(2,\frac{8}{3})$ and the existence of the nonlocal term
$\int_{\mathbb{R}^3}\phi_uu^2\,dx$, the method introduced in
\cite[Lemma 2.6]{JJLV2022} becomes invalid.
In fact, if we do the scaling $v=\sqrt{\theta}u$, it is impossible to obtain that
\[
 m(\theta\alpha)\leq \theta m(\alpha),\quad \theta>1, \; \alpha>0.
\]

Following \cite{Bellazziniscal}, we introduce the condition
\begin{equation}
\text{the function $c\to\frac{m(c)}{c}$ is strictly decreasing}. \label{eMD}
\end{equation}
From this assumption it follows
\[
 \frac{c_1}{c}m(c)<m(c_1),\quad \frac{c-c_1}{c}<m(c-c_1).
\]
That is,
\[
 m(c)=\frac{c_1}{c}m(c)+\frac{c-c_1}{c}m(c)<m(c_1)+m(c-c_1),
\]
whenever $0<c_1<c$. However, verifying condition
\eqref{eMD} is not easy since the function $c\to\frac{m(c)}{c}$ has oscillatory behavior,
even in a neighborhood of the origin. To overcome this difficulty, we adapt
the techniques developed by Bellazzini-Siciliano \cite{Bellazziniscal,Bellazzini2011}
as to finding sufficient conditions to avoid dichotomy.

As we see, the presence of the term $\int_{\mathbb{R}^3}\phi_uu^2\,dx$ has a
significant impact on the geometry of $\Phi_u(t)$ and on the existence of ground
state solutions. The existence of a normalized ground state solution for the
nonlinear Schr\"odinger equation with a Sobolev critical term and a $L^2$-subcritical
perturbation was discussed in \cite{JJLV2022}, where the local minima of the
constraint functional is exactly a ground state.
However, it is not trivial for our case due to the complex structure of the fiber
map. For this situation, we turn our attention to the Pohozaev manifold
$\Lambda(c)$ which contains the solutions of \eqref{SBP},
see Lemma \ref{npidentity}, to search for a ground state by taking the minimization
in the set of solutions.
Our main result reads as follows.

\begin{theorem}\label{thm: main result}
Let $p\in(2,8/3)$. For any $\mu>0$ there exists a $c_0=c_0(\mu)>0$
such that, for any $c\in(0,c_0)$, $I(u)$ restricted to $S(c)$ has a critical point
$u_c$ at a negative level $m(c)<0$ which is an interior local minimizer of $I(u)$
in the set $V(c)$. Moreover, system \eqref{SBP} admits a ground state solution on
$S(c)$.
\end{theorem}

Before concluding this section, we would like to summarize new features in this
study.
\begin{itemize}
 \item The approach which we use for Theorem \ref{thm: main result} distinguishes
from those described in the literature, for example, see \cite[Theorem 1.2]{LZ2023}.
In fact, our arguments are based on the minimizing method instead of the mountain
pass theorem. Moreover, our arguments work for all $\mu>0$ and we do not need to
assume the range of the Lagrange multiplier $\lambda$ in the first step.
 \item To show that the minimizing sequences for $m(c)$ are relatively compact,
we make use of the strong subadditivity inequality \eqref{strict add}. We can not
just take the same steps as shown in \cite{JJLV2022} because of the presence of the
term $\int_{\mathbb{R}^3}\phi_uu\,dx$. Therefore, we turn to develop another method
to ensure the inequality \eqref{strict add} to be true. Namely, we take into account
how to guarantee the condition \eqref{eMD}.

 \item The exponential term in $\phi_u$ makes the fiber map $\Phi_u(t)$ more
complicated since it exists in the first and second derivatives. Thus we cannot
follow \cite{JJLV2022} directly to draw a conclusion that any ground state is
contained in $V(c)$. To show the existence of the ground state solution, we take
a series of solutions in $\Lambda(c)$ and obtain the bounded Palais-Smale sequences
which are the minimizers of $I(u)$ on $S(c)$.
\end{itemize}

This article is organized as follows.
In Section \ref{preliminary}, we present some preliminaries and lemmas that will
be used later.
In Section \ref{1}, we prove our main result by clarifying the local minima
structure and showing the convergence, up to a translation, of all minimizing
sequences for the functional $I(u)$ on $V(c)$.

\section{Preliminary results}\label{preliminary}

Throughout this paper,  for any $1\leq s <\infty$, we denote by $L^s(\mathbb{R}^3)$
the usual Lebesgue space with norm $\|u\|^s_s:=\int_{\mathbb{R}^3}|u|^s\,dx$.
 We use $C_0^{\infty}(\mathbb{R}^3)$ to denote the space of the functions
infinitely differentiable with compact support in $\mathbb{R}^3$.
We denote by $C_1, C_2,\dots$ the positive constants whose values possibly
vary from line to line. The open ball in $\mathbb{R}^3$ is denoted by $B(x,R)$
with the center at $x$ and the radius $R$.

We start with the Hardy-Littlewood-Sobolev inequality \cite{liebloss}:
\begin{equation}\label{hls}
 \big|\int_{\mathbb{R}^N}\int_{\mathbb{R}^N}\frac{f(x)g(y)}{|x-y|^{\lambda}}
\,dx\,dy\big|\leq C(N,\lambda,p,q)\|f\|_{p}\|g\|_{q},
\end{equation}
where $f\in L^p(\mathbb{R}^N)$, $g\in L^q(\mathbb{R}^N), p,q>1$,
$0<\lambda<N$, and $\frac{1}{p}+\frac{1}{q}+\frac{\lambda}{N}=2$.

The following Gagliardo-Nirenberg inequality can be found in \cite{1983wein}:
\begin{equation}\label{gn}
 \|u\|_p\leq K_{GN}^{1/p}\|\nabla u\|_2^{\tau}\|u\|_2^{1-\tau},
\end{equation}
where $N\ge3$, $p\in[2,\frac{2N}{N-2}]$, and $\tau=N(\frac{1}{2}-\frac{1}{p})$.

We recall the optimal Sobolev constant $\mathcal{S}>0$, see \cite{HB1983}, which
is
\[
 \mathcal{S}=\inf_{ u\in\mathcal{D}^{1,2}(\mathbb{R}^3),\; u\neq 0}
\frac{\|\nabla u\|_2^2}{\|u\|^2_6},
\]
where
\[
 \mathcal{D}^{1,2}(\mathbb{R}^3):=\{u\in L^6(\mathbb{R}^3): |\nabla u|\in L^2(\mathbb{R}^3)\}
\]
is the completion of $C_0^{\infty}(\mathbb{R}^3)$ with the norm
\[
 \|u\|_{\mathcal{D}^{1,2}(\mathbb{R}^3)}=\|\nabla u\|_2.
\]
The Hilbert space defined by
\[
 \mathcal{D}:=\{u\in\mathcal{D}^{1,2}(\mathbb{R}^3):\Delta u\in L^2(\mathbb{R}^3)\}
\]
is the completion of $C_0^{\infty}(\mathbb{R}^3)$ with respect to the norm
\[
 \|u\|_{\mathcal{D}}^2=\|\Delta u\|_2^2+\|\nabla u\|_2^2.
\]
It is easy to show that $\mathcal{D}$ is continuously embedded into
$\mathcal{D}^{1,2}(\mathbb{R}^3)$, see \cite{DS2019}.

\begin{lemma}\cite[Lemma 3.4]{DS2019}\label{nonlocal property}
 For every $u\in H^1(\mathbb{R}^3)$ we have
 \begin{enumerate}
 \item[(i)] for every $y\in\mathbb{R}^3$, $\phi_{u(\cdot+y)}=\phi_u(\cdot+y)$;
 \item[(ii)] $\phi_u\ge0$;
 \item[(iii)] $\phi_u\in\mathcal{D}$;
 \item[(iv)] $\|\phi_u\|_6\leq C\|u\|_{H^1(\mathbb{R}^3)}^2$; and
 \item[(v)] if $v_n\rightharpoonup v$ in $H^1(\mathbb{R}^3)$, then $\phi_{v_n}\rightharpoonup\phi_{v}$ in $\mathcal{D}$.
 \end{enumerate}
\end{lemma}

\begin{lemma}\label{CDestimate}
 Let $u\in S(c)$. Then we have
 \begin{enumerate}
 \item[(i)] There exists a constant $K_H>0$ such that
 \[
 \int_{\mathbb{R}^3}\int_{\mathbb{R}^3}\frac{|u(x)|^2|u(y)|^2}{|x-y|}\,dx\,dy\leq K_H\|\nabla u\|_2\,c^{3/2}.
 \]
 \item[(ii)] There exists a constant $K_{GN}>0$ such that
 \[
 \|u\|_p^p\leq K_{GN}\|\nabla u\|_2^{\frac{3(p-2)}{2}}c^{\frac{6-p}{4}}.
 \]
 \end{enumerate}
\end{lemma}

\begin{proof}
From \eqref{hls} and \eqref{gn} it follows that
 \[
 \int_{\mathbb{R}^3}\int_{\mathbb{R}^3}\frac{|u(x)|^2|u(y)|^2}{|x-y|}\,dx\,dy
\leq K_1\|u\|_{12/5}^4\leq K_H\|\nabla u\|_{2}\|u\|_{2}^3,
 \]
which implies (i).
 In view of \eqref{gn}, we derive
 \[
 \|u\|_{p}^p\leq K_{GN}\|\nabla u\|_2^{\frac{3(p-2)}{2}}\|u\|_2^{\frac{6-p}{2}},
 \]
 which leads to (ii).
\end{proof}

\begin{lemma}\label{npidentity}
 Let $\mu>0$ and $2<p<6$. If $(u, \lambda)\in H^1(\mathbb{R}^3)\times \mathbb{R}$
 weakly solves
 \begin{equation}\label{npsbpeq}
 -\Delta u+\phi_u u=\lambda u+\mu|u|^{p-2}u+|u|^4u,
 \end{equation}
then $Q(u)=0$, where $Q(u)$ is defined by \eqref{pohozaevtypeid}.
\end{lemma}

\begin{proof}
Using the Pohozaev identity in \cite[Lemma 4.2]{CTcritical2020} yields
 \begin{equation}\label{pohoid}
\begin{aligned}
 &\frac{1}{2}\int_{\mathbb{R}^3}|\nabla u|^2\,dx
 +\frac{5}{4}\int_{\mathbb{R}^3}\phi_uu^2\,dx
 +\frac{1}{16\pi}\int_{\mathbb{R}^3}\int_{\mathbb{R}^3}e^{-|x-y|}u^2(x)u^2(y)
 \,dx\,dy\\
 &=\frac{3\lambda}{2}\int_{\mathbb{R}^3}|u|^2\,dx
 +\frac{3\mu}{p}\int_{\mathbb{R}^3}|u|^p\,dx
 +\frac12\int_{\mathbb{R}^3}|u|^6\,dx=0.
 \end{aligned}
\end{equation}
 Multiplying both sides of \eqref{npsbpeq} by $u$ and integrating on $\mathbb{R}^3$
leads to
 \begin{equation}\label{neharid}
 \int_{\mathbb{R}^3}|\nabla u|^2\,dx+\int_{\mathbb{R}^3}\phi_uu^2\,dx
=\lambda\int_{\mathbb{R}^3}|u|^2\,dx+\mu\int_{\mathbb{R}^3}|u|^p\,dx
+\int_{\mathbb{R}^3}|u|^6\,dx.
 \end{equation}
By combining \eqref{pohoid} and \eqref{neharid}, we obtain
 \begin{align*}
 &\int_{\mathbb{R}^3}|\nabla u|^2\,dx
+\frac{1}{16\pi}\int_{\mathbb{R}^3}\int_{\mathbb{R}^3}
 \frac{1-e^{-|x-y|}}{|x-y|}u^2(x)u^2(y)\,dx\,dy\\
 &=\frac{1}{16\pi}\int_{\mathbb{R}^3}\int_{\mathbb{R}^3}e^{-|x-y|}u^2(x)u^2(y)\,dx\,dy
+\frac{3\mu\,(p-2)}{2p} \int_{\mathbb{R}^3}|u|^p\,dx+\int_{\mathbb{R}^3}|u|^6\,dx.
 \end{align*}
Therefore, we arrive at the desired result.
\end{proof}

To ensure the condition \eqref{eMD}, we define
\begin{gather*}
A(u):=\int_{\mathbb{R}^3}|\nabla u|^2\,dx,\quad
B(u):=\int_{\mathbb{R}^3}\phi_uu^2\,dx, \\
C(u):=\int_{\mathbb{R}^3}|u|^p\,dx, \quad
D(u):=\int_{\mathbb{R}^3}|u|^6\,dx, \\
T(u):=\frac{1}{4}B(u)-\frac{\mu}{p}C(u)-\frac{1}{6}D(u).
\end{gather*}
Thus, the functional $I(u)$ can be simply re-written as
\[
I(u)=\frac{1}{2}A(u)+T(u).
\]
Next we recall two definitions introduced in \cite{Bellazziniscal}.

\begin{definition}\label{md condition defi1} \rm
 Let $u\in H^1(\mathbb{R}^3)$ with $u\neq0$. A continuous path
$g_u:\theta\in\mathbb{R}^+\mapsto g_u(\theta)\in H^1(\mathbb{R}^3)$ such that
 $g_u(1)=u$ is said to be a scaling path of $u$ if
$\Theta_{g_u}(\theta):=\|g_u(\theta)\|_2^2/\|u\|_2^2$ is differentiable
and $\Theta'_{g_u}(1)\neq0$. We denote by $G_u$ the set of the scaling paths of $u$.
\end{definition}

\begin{definition}\label{md condition defi2} \rm
 Let $u\neq0$ be fixed and $g_u\in G_u$. We say that the scaling path $g_u$ is
admissible for the functional $I$ if
 \[
 h_{g_u}(\theta):=I(g_u(\theta))-\Theta_{g_u}(\theta)I(u),~\theta\ge0
 \] is a differentiable function.
\end{definition}

The following lemma is regarding the splitting properties of the term $T$,
see \cite[Lemma 2.8]{HLC2022} and \cite[Lemma B.2]{DS2019}.

\begin{lemma}\label{split property of T}
 If $p\in(2,10/3)$, we let $\{u_n\}\subset V(c)$ be a minimizing sequence
for $m(c)$ such that $u_n\rightharpoonup u\neq0$. Then $T$ satisfies the following
properties:
 \begin{enumerate}
 \item[(i)] $T(u_n-u)+T(u)=T(u_n)+o_n(1)$; and
 \item[(ii)] $T(\alpha_n(u_n-u))-T(u_n-u)=o_n(1)$, where
$\alpha_n=\|u_n\|_2^2-\|u\|_2^2/\|u_n-u\|_2^2$.
 \end{enumerate}
\end{lemma}

The following proposition provides us a useful criterion for the
 condition \eqref{eMD}, \cite[Theorem 2.1]{Bellazziniscal}.

\begin{proposition}\label{strong subadd}
Let $T\in C^1(H^1(\mathbb{R}^3),\mathbb{R})$ satisfy Lemma \ref{split property of T}
(i) and (ii). Assume that for every $c>0$, all the minimizing sequences
$\{u_n\}$ for $m(c)$ have a weak limit up to translations different from zero.
Assume that \eqref{weak add} and the following conditions hold
 \begin{gather}\label{md condi1}
 -\infty<m(c)<0,\quad \text{for all }c>0\; (I(0)=0), \\
 \label{md condi2}
 c\mapsto m(c)\quad \text{is continuous}, \\
\label{md condi3}
 \lim_{c\to0}\frac{m(c)}{c}=0.
 \end{gather}
Then, for each $c>0$, the set
 \[
 M(c)=\cup_{\tilde{c}\in(0,c]}\{u\in S(\tilde{c}): I(u)=m(\tilde{c})\}
 \]
is nonempty. In addition, if
 \begin{equation}\label{md condi4}
 \forall u\in M(c),\; \exists g_u\in G_u\text{ is admissible such that }
\frac{d}{d\theta}h_{g_u}(\theta)|_{\theta=1}\neq0,
 \end{equation}
then the condition \eqref{eMD} holds. Moreover, if $\{u_n\}$ is a minimizing
sequence weakly convergent to a certain $u$~(necessarily~$\neq$~0),
then $\|u_n-u\|_{H^1(\mathbb{R}^3)}\rightarrow0$ and $I(u)=m(c)$.
\end{proposition}

\section{Proof of Theorem \ref{thm: main result}}\label{1}

Throughout the whole section, we  assume that $2<p<\frac{8}{3}$, from which we
have $0<\frac{3(p-2)}{2}<1$.
Note that
\begin{equation}\label{functional pro ge}
 I(u)\ge\frac{1}{2}\|\nabla u\|_2^2-\frac{\mu K_{GN}}{p}c^{\frac{6-p}{4}}\|\nabla u\|_2^{\frac{3(p-2)}{2}}-\frac{1}{6\mathcal{S}^3}\|\nabla u\|_2^6
\end{equation}
for any $u\in S(c)$. Now we consider the function $h:\mathbb{R}^+\to\mathbb{R}$, defined by
\begin{align*}
 h_c(t)&:=\frac{1}{2}t^2-\frac{\mu K_{GN}}{p}c^{\frac{6-p}{4}}
 t^{\frac{3(p-2)}{2}}-\frac{1}{6\mathcal{S}^3}t^6\\
 &=t^2[\frac{1}{2}-\frac{\mu K_{GN}}{p}c^{\frac{6-p}{4}}
 t^{\frac{3(p-2)}{2}-2}-\frac{1}{6\mathcal{S}^3}t^4].
\end{align*}
Since $\mu>0$ and $\frac{3(p-2)}{2}<1$, we have $h_c(0^+)=0^-$ and
$h_c(+\infty)=-\infty$. Moreover, the following properties hold
for $h_c$.

\begin{lemma}[{\cite[Lemma 2.1]{JJLV2022}}] \label{functionh}
For any $\mu>0$ there exist a $c_0=c_0(\mu)>0$ and $\rho_0:=\rho_{c_0}>0$
such that $h_{c_0}(\sqrt{\rho_0})=0$ and $h_c(\sqrt{\rho_0})>0$ hold for any
$c\in(0,c_0)$, where $c_0$ and $\rho_0$ are explicitly given by
 \[
 c_0:=\big(\frac{1}{2K}\big)^{3/2}>0,
 \]
 with
 \begin{gather*}
\begin{aligned}
 K&:=\frac{\mu}{p}K_{GN}\Big[-\frac{3(3p-10)\mu K_{GN}\mathcal{S}^3}{4p}
 \Big]^{\frac{3p-10}{3(6-p)}}
+\frac{1}{6\mathcal{S}^3}\Big[-\frac{3(3p-10)\mu K_{GN}\mathcal{S}^3}{4p}
 \Big]^{\frac{8}{3(6-p)}}\\
&\quad >0,
\end{aligned}\\
 \rho_0:=\Big[-\frac{3(3p-10)\mu K_{GN}\mathcal{S}^3}{4p}\Big]^{\frac{4}{3(6-p)}}
 c_0^{1/3}.
 \end{gather*}
\end{lemma}

\begin{lemma}[{\cite[Lemma 2.2]{JJLV2022}}] \label{functionh1}
 Let $(c_1,\rho_1)\in(0,\infty)\times(0,\infty)$ satisfy
$h_{c_1}(\sqrt{\rho_1})\ge0$. Then for any $c_2\in(0,c_1]$ there holds
 \[
 h_{c_2}(\sqrt{\rho_2})\ge0,\quad \text{if }
\rho_2\in\big[\frac{c_2}{c_1}\rho_1,\rho_1\big].
 \]
\end{lemma}

\begin{remark} \rm
 For $p\in(2,10/3)$, from the expression of $h_c(t)$ we can deduce that
$h_c(0^+)=0^-$ and $h_c(+\infty)=-\infty$, which means that
Lemma \ref{functionh} also holds in such a case. However, taking into account
the geometrical structure of the fiber map $\Phi_u(t)$, we have to reduce the
 range of $p$ to $p\in(2,\frac{8}{3})$, see Lemma \ref{mcproperty} below.
\end{remark}

\begin{remark}\label{remark positive} \rm
According to Lemmas \ref{functionh} and \ref{functionh1}, it is not difficult to
see that $h_{c_0}(\sqrt{\rho_0})=0$ and $h_{c}(\sqrt{\rho_0})>0$ for all
$c\in(0,c_0)$.
\end{remark}

\begin{lemma}\label{npidentity coer}
 For $c\in(0,c_0)$, $I(u)$ restricted to $\Lambda(c)$ is coercive on
$H^1(\mathbb{R}^3)$. Namely, if $\{u_n\}\subset H^1(\mathbb{R}^3)$ satisfies
$\|u_n\|_{H^1(\mathbb{R}^3)}\to+\infty$, then $I(u_n)\to+\infty$.
\end{lemma}

\begin{proof}
Let $u\in\Lambda(c)$. Taking into account  $Q(u)=0$, we have
\begin{equation}\label{npidentity coer eq1}
\begin{aligned}
&A(u)+\frac{1}{4}B(u)-\frac{1}{16\pi}\int_{\mathbb{R}^3}
\int_{\mathbb{R}^3}e^{-|x-y|}u^2(x)u^2(y)\,dx\,dy\\
&=\frac{3\mu(p-2)}{2p}C(u)+D(u).
\end{aligned}
\end{equation}

From Lemma \ref{CDestimate} (i), there exists a constant $C_1>0$ such that
 \[
 B(u)\leq C_1\|\nabla u\|_2c^{3/2}.
 \]
In view of \eqref{npidentity coer eq1}, there exists a constant $C_2>0$ such that
 \[
D(u)\leq A(u)+C_2\|\nabla u\|_2c^{3/2}.
 \]
This together with Lemma \ref{CDestimate} (ii), for some $C_3>0$, leads to
 \begin{align*}
 I(u)&=\frac{1}{2}A(u)+\frac{1}{4}B(u)-\frac{\mu}{p}C(u)-\frac{1}{6}D(u)\\
 &\ge\frac{1}{2}A(u)-\frac{\mu}{p}C(u)-\frac{1}{6}A(u)-\frac{C_2}{6}A(u)^{\frac{1}{2}}c^{3/2}\\
 &\ge\frac{1}{3}A(u)-C_3A(u)^{\frac{3(p-2)}{4}}c^{\frac{6-p}{4}}-\frac{C_2}{6}A(u)^{\frac{1}{2}}c^{3/2},
 \end{align*}
from which we e complete the proof.
\end{proof}

Define
\[
 B_{\rho_0}:=\{u\in H^1(\mathbb{R}^3):\|\nabla u\|_2^2<\rho_0\},\quad
V(c):=S(c)\cap B_{\rho_0}
\]
and consider a minimization problem:
\[
 m(c)=\inf_{u\in V(c)} I(u),~~\text{for any}~~c\in(0,c_0).
\]

\begin{lemma}\label{mcproperty}
 Let $c\in(0,c_0)$. Then the following three assertions hold.
 \begin{enumerate}
 \item[(i)] $m(c)=\inf_{u\in V(c)} I(u)<0<\inf_{u\in\partial V(c)}I(u)$.
 \item[(ii)] The function $c\mapsto m(c)$ is a continuous mapping.
 \item[(iii)] For all $\alpha\in(0,c)$, we have $m(c)\leq m(\alpha)+m(c-\alpha)$.
 \end{enumerate}
\end{lemma}

\begin{proof}
(i) For any $u\in\partial V(c)$ we have $\|\nabla u\|_2^2=\rho_0$.
 From \eqref{functional pro ge} it follows that
 \[
  I(u)\ge h_c(\|\nabla u\|_2)=h_c(\sqrt{\rho_0})>0.
 \]
Taking into account $\frac{3(p-2)}{2}<1$, we have $\Phi_u(t)\to0^-$ as $t\to0$.
Therefore, there exists $t_0<1$ small enough such that
$\|\nabla u^{t_0}\|_2^2=t_0^2\|\nabla u\|_2^2<\rho_0$ and $I(u^{t_0})=\Phi_u(t_0)<0$,
which means $m(c)<0$.

(ii) For any $c\in(0,c_0)$, let $\{c_n\}\subset (0,c_0)$ satisfy $c_n\to c$ as
$n\to\infty$. Recall the definition of $m(c_n)<0$. For any $\epsilon>0$ small enough,
there exists $\{u_n\}\subset V(c_n)$ such that
 \[
  I(u_n)\leq m(c_n)+\epsilon,\ ~~I(u_n)<0.
 \]
Let $v_n:=\big(\frac{c}{c_n}\big)^{\sqrt{1/2}}u_n$. Then $v_n\in S(c)$ and by similar
arguments as described in \cite[Lemma 2.6]{JJLV2022}, we see that $v_n\in V(c)$.
Furthermore, we have
 \begin{align*}
  m(c)&\leq I(v_n)\\
&=\frac{1}{2}\frac{c}{c_n}A(u_n)+\frac{1}{4}(\frac{c}{c_n})^2B(u_n)-\frac{\mu}{p}(\frac{c}{c_n})^{\frac{p}{2}}C(u_n)-\frac{1}{6}(\frac{c}{c_n})^3D(u_n)\\
&=I(u_n)+o_n(1)\\
&\leq m(c_n)+\epsilon+o_n(1).
 \end{align*}
Similarly, for any $\epsilon>0$ there exists $u\in V(c)$ such that
 \[
  I(u)\leq m(c)+\epsilon,~~I(u)<0.
 \]
Let $v_n:=\big(\frac{c_n}{c}\big)^{\sqrt{1/2}}u$. Then $v_n\in V(c_n)$. Processing
in an analogous manner, we can obtain
 \[
  m(c_n)\leq I(u_n)=[I(u_n)-I(u)]+I(u)\leq m(c)+\epsilon+o_n(1).
 \]
In view of $\epsilon>0$ being arbitrary, we have $m(c_n)\to m(c)$ as $n\to\infty$
which implies (ii).

(iii) By the fact that $C_0^{\infty}(\mathbb{R}^3)$ is dense in $H^1(\mathbb{R}^3)$,
for any $\epsilon>0$ there exist $u_1\in C_0^{\infty}(\mathbb{R}^3)\cap V(\alpha)$ and
$u_2\in C_0^{\infty}(\mathbb{R}^3)\cap V(c-\alpha)$ satisfying
 \begin{gather}\label{Iu1u2 eq1}
I(u_1)\leq m(\alpha)+\frac{\epsilon}{2},\quad
I(u_2)\leq m(c-\alpha)+\frac{\epsilon}{2}, \\
\label{Iu1u2}
  I(u_1)<0,\quad I(u_2)<0.
\end{gather}
Moreover, for any $n\in\mathbb{N}$, by a translation, we may assume that
 \[
\operatorname{dist}(\operatorname{supp} u_1, \operatorname{supp} u_2)>n.
 \]
 By Lemma \ref{functionh1}, we have $h_{\alpha}(\sqrt{\rho})\ge0$
for any $\rho\in[\frac{\alpha}{c}\rho_0,\rho_0]$ and
$h_{c-\alpha}(\sqrt{\rho})\ge 0$ for each
$\rho\in[\frac{c-\alpha}{c}\rho_0,\rho_0]$. Hence, we can deduce from
\eqref{Iu1u2} that
 \[
  \|\nabla u_1\|_2^2<\frac{\alpha}{c}\rho_0,\quad
\|\nabla u_2\|_2^2<\frac{c-\alpha}{c}\rho_0.
 \]
Let $u=u_1+u_2$. It is easy to verify that $\|u\|_2^2=\|u_1\|_2^2+\|u_2\|_2^2$
and $A(u)=A(u_1)+A(u_2)$. Thus we have
 \[
  \|u\|_2^2=c,\ ~~~\|\nabla u\|_2^2<\rho_0,
 \]
which means $u\in V(c)$.

Notice that
 \begin{align*}
  |u(x)|^2|u(y)|^2
&=|u_1(x)+u_2(x)|^2|u_1(y)+u_2(y)|^2\\
&=|u_1(x)|^2|u_1(y)|^2+|u_1(x)|^2|u_2(y)|^2
+|u_2(x)|^2|u_1(y)|^2 \\
&\quad +|u_2(x)|^2|u_2(y)|^2 +2|u_1(x)|^2u_1(y)u_2(y)+2|u_2(x)|^2u_1(y)u_2(y)\\
&\quad+2|u_1(y)|^2u_1(x)u_2(x)+2|u_2(y)|^2u_1(x)u_2(x)\\
&\quad +4u_1(x)u_2(x)u_1(y)u_2(y).
 \end{align*}
Then we can deduce that
 \begin{align*}
\int_{\mathbb{R}^3}\int_{\mathbb{R}^3}\frac{1-e^{-|x-y|}}{|x-y|}|u_1(x)|^2|u_2(y)|^2\,dx\,dy
&=\int_{\mathbb{R}^3}\int_{\mathbb{R}^3}\frac{1-e^{-|x-y|}}{|x-y|}|u_1(y)|^2|u_2(x)|^2\,dx\,dy\\
&\leq\int_{\mathbb{R}^3}\int_{\mathbb{R}^3}\frac{|u_1(y)|^2|u_2(x)|^2}{|x-y|}\,dx\,dy\\
&=\int_{\operatorname{supp} u_1}\int_{\operatorname{supp} u_2}\frac{|u_1(y)|^2|u_2(x)|^2}{|x-y|}\,dx\,dy\\
&\leq\frac{\alpha(c-\alpha)}{n},
 \end{align*}
 \begin{align*}
&\int_{\mathbb{R}^3}\int_{\mathbb{R}^3}
 \frac{1-e^{-|x-y|}}{|x-y|}|u_1(x)|^2|u_1(y)||u_2(y)|\,dx\,dy\\
&= \int_{\mathbb{R}^3}\int_{\mathbb{R}^3}
 \frac{1-e^{-|x-y|}}{|x-y|}|u_1(y)|^2|u_1(x)||u_2(x)|\,dx\,dy
\leq\frac{\alpha c}{2n},
 \end{align*}
\begin{align*}
&\int_{\mathbb{R}^3}\int_{\mathbb{R}^3}
 \frac{1-e^{-|x-y|}}{|x-y|}|u_2(x)|^2|u_1(y)||u_2(y)|\,dx\,dy\\
&=\int_{\mathbb{R}^3}\int_{\mathbb{R}^3}
 \frac{1-e^{-|x-y|}}{|x-y|}|u_2(y)|^2|u_1(x)||u_2(x)|\,dx\,dy
\leq\frac{(c-\alpha) c}{2n},
\end{align*}
and
\begin{align*}
&\int_{\mathbb{R}^3}\int_{\mathbb{R}^3}
\frac{1-e^{-|x-y|}}{|x-y|}|u_1(x)||u_2(x)||u_1(y)||u_2(y)|\,dx\,dy\\
&\leq \int_{\mathbb{R}^3}\int_{\mathbb{R}^3}
\frac{|u_1(x)||u_2(x)||u_1(y)||u_2(y)|}{|x-y|}\,dx\,dy \leq\frac{c^2}{4n}.
 \end{align*}
Therefore,
\begin{align*}
  B(u)=B(u_1)+B(u_2)+o_n(1).
\end{align*}
Clearly, it holds
\[
  C(u)=C(u_1)+C(u_2),\quad D(u)=D(u_1)+D(u_2).
\]
Therefore, from \eqref{Iu1u2 eq1} it follows that
 \begin{align*}
  m(c)
&\leq I(u)=I(u_1)+I(u_2)+o_n(1)\\
&\leq m(\alpha)+m(c-\alpha)+\epsilon+o_n(1).
 \end{align*}
Since $\epsilon>0$ is arbitrary, we have $m(c)\leq m(\alpha)+m(c-\alpha)$.
Consequently, we have arrived at (iii).
\end{proof}

\begin{lemma}\label{nonvanish pro1}
Let $\{v_n\}\subset B_{\rho_0}$ satisfy $\|v_n\|_p\to0$. Then there exists a
$\beta_0>0$ such that
\[
 I(v_n)\ge\beta_0\|\nabla v_n\|_2^2+o_n(1).
\]
\end{lemma}

\begin{proof}
 By a direct calculation, we have
 \begin{align*}
 I(v_n)
&\ge\frac{1}{2}\|\nabla v_n\|_2^2-\frac{1}{6}\|v_n\|_6^6+o_n(1)\\
&\ge\frac{1}{2}\|\nabla v_n\|_2^2-\frac{1}{6\mathcal{S}^3}\|\nabla v_n\|_2^6+o_n(1)\\
&\ge\|\nabla v_n\|_2^2[\frac{1}{2}-\frac{1}{6\mathcal{S}^3}\rho_0^2]+o_n(1).
 \end{align*}
Since $h_{c_0}(\sqrt{\rho_0})=0$, we obtain
 \[
 \beta_0:=\frac{1}{2}-\frac{1}{6\mathcal{S}^3}\rho_0^2
=\frac{\mu K_{GN}}{p}c_0^{\frac{6-p}{4}}\rho_0^{\frac{3(p-2)}{4}-1}>0. \qedhere
 \]
\end{proof}

\begin{lemma}\label{nonvanish pro}
 For any $c\in(0,c_0)$, let $\{u_n\}\subset V(c)$ be a minimizing sequence for
$m(c)$ such that $u_n\rightharpoonup u$ in $H^1(\mathbb{R}^3)$ as $n\to\infty$.
Then $u\neq0$.
\end{lemma}

\begin{proof}
To show that there exist a $\beta_1>0$ and a sequence $\{y_n\}\subset\mathbb{R}^3$
such that for some $R>0$ it holds
 \begin{equation}\label{nonvanish proeq1}
 \int_{B(y_n,R)}|u_n|^2\,dx\ge\beta_1>0,
 \end{equation}
we argue by contradiction that \eqref{nonvanish proeq1} does not hold.
 According to \cite[Lemma I.1]{PLION1984}, for $2<p<6$ we have
 \[
 \|u_n\|_{L^p(\mathbb{R}^3)}\to0\quad \text{as }n\to\infty.
 \]
Then, from Lemma \ref{nonvanish pro1} it follows that $I(u_n)\ge o_n(1)$.
This contradicts the fact that $m(c)<0$.

From \eqref{nonvanish proeq1}, we know that there exist some $C>0$ and a sequence
$\{y_n\}\subset\mathbb{R}^3$ such that
 \[
 \int_{B(0,R)}|u_n(\cdot-y_n)|^2\,dx\ge C>0.
 \]
Because of the Rellich compactness theorem, we have
 \[
 u_n(x-y_n)\rightharpoonup u\neq0~~~\text{in}~~~H^1(\mathbb{R}^3),
 \]
which enables us to arrive at the desired result.
\end{proof}

% page 13

Next we are going to verify all conditions described in
Proposition \ref{strong subadd} for presenting the proof of
Theorem \ref{thm: main result}.
In view of Lemmas \ref{split property of T}, \ref{mcproperty},
and \ref{nonvanish pro}, it suffices to prove \eqref{md condi3} and
\eqref{md condi4}.

\begin{lemma}\label{strong add con3}
 Assume that $c\in(0,c_0)$. Then the function $c\mapsto m(c)$ satisfies
\eqref{md condi3}.
\end{lemma}

\begin{proof}
 Since $m(c)<0$, we have
 \[
 \frac{\tilde{m}(c)}{c}\leq\frac{m(c)}{c}<0,
 \]
 where
 \[
\tilde{m}(c):=\inf_{u\in V(c)}\tilde{I}(u), \quad
\tilde{I}(u):=\frac{1}{2}\int_{\mathbb{R}^3}|\nabla u|^2\,dx
-\frac{\mu}{p}\int_{\mathbb{R}^3}|u|^p\,dx-\frac16\int_{\mathbb{R}^3}|u|^6\,dx.
 \]
Thus, it suffices to show that $\frac{\tilde{m}(c)}{c}\to0$ as $c\to0$.
Indeed, $\tilde{I}(u)$ is the functional associated with the following
Schr\"odinger equation with combined nonlinearities
 \[
 -\Delta u=\lambda u+\mu|u|^{p-2}u+|u|^4u~~\text{in}~~\mathbb{R}^3,
 \]
with the normalized condition $\|u\|_2^2=c$.
According to \cite[Theorem 1.2]{JJLV2022}, for any $c\in(0,c_0)$ there exists
$u_c\in V(c)$ such that $\tilde{m}(c)=\tilde{I}(u_c)<0$. Moreover, we know
that the sequence $\{u_c\}_{c\in(0,c_0)}$ is bounded in
$\mathcal{D}^{1,2}(\mathbb{R}^3)$ as $c\to0$ and $u_c$ satisfies the following
equation in the weak sense
 \begin{equation}\label{strong add con3 equation}
 -\Delta u_c=\lambda_c u_c+\mu|u_c|^{p-2}u_c+|u_c|^4u_c~~\text{in}~~\mathbb{R}^3,
 \end{equation}
from which we deduce that
 \begin{align*}
 \frac{\lambda_c}{2}
&=\frac{\int_{\mathbb{R}^3}|\nabla u_c|^2\,dx-\mu\int_{\mathbb{R}^3}|u_c|^p\,dx
 -\int_{\mathbb{R}^3}|u_c|^6\,dx}{2\int_{\mathbb{R}^3}|u_c|^2\,dx}\\
&\leq\frac{\frac{1}{2}\int_{\mathbb{R}^3}|\nabla u_c|^2\,dx
 -\frac{\mu}{p}\int_{\mathbb{R}^3}|u_c|^p\,dx
 -\frac{1}{6}\int_{\mathbb{R}^3}|u_c|^6\,dx}{\int_{\mathbb{R}^3}|u_c|^2\,dx}\\
&=\frac{\tilde{I}(u_c)}{c}=\frac{\tilde{m}(c)}{c}<0.
 \end{align*}

To show that $\lim_{c\to0}\lambda_c=0$, we argue by contradiction: assume that
there exists a sequence $c_n\to0$ such that $\lambda_{c_n}<-C$ for some $C\in(0,1)$.
Since the minimizers $u_n:=u_{c_n}$ satisfies \eqref{strong add con3 equation},
there exist some $C_1, C_2>0$ such that
 \begin{align*}
 C\|u_n\|_{H^1(\mathbb{R}^3)}^2
&\leq\int_{\mathbb{R}^3}|\nabla u_n|^2\,dx+C\int_{\mathbb{R}^3}|u_n|^2\,dx\\
&<\int_{\mathbb{R}^3}|\nabla u_n|^2\,dx-\lambda_{c_n}\int_{\mathbb{R}^3}|u_n|^2\,dx\\
&=\mu\int_{\mathbb{R}^3}|u_n|^p\,dx+\int_{\mathbb{R}^3}|u_n|^6\,dx\\
&\leq C_1\|u_n\|_{H^1(\mathbb{R}^3)}^{p}+C_2\|u_n\|_{H^1(\mathbb{R}^3)}^6.
 \end{align*}
This indicates that there exists $C_3>0$ such that $\|\nabla u_n\|_2^2>C_3>0$
because $p>2$. Hence, from Remark \ref{remark positive} it follows that
 \begin{align*}
 0&>\tilde{I}(u_n) \\
&\ge\frac{1}{2}\|\nabla u_n\|_2^2-\frac{\mu K_{GN}}{p}c_n^{\frac{6-p}{4}}\|\nabla u_n\|_2^{\frac{3(p-2)}{2}}-\frac{1}{6\mathcal{S}^3}\|\nabla u_n\|_2^6\\
&=\|\nabla u_n\|_2^2\Big(\frac{1}{2}-\frac{\mu K_{GN}}{p}c_n^{\frac{6-p}{4}}
 \|\nabla u_n\|_2^{\frac{3(p-2)}{2}-2}
 -\frac{1}{6\mathcal{S}^3}\|\nabla u_n\|_2^4\Big)\\
&\ge C_3\Big(\frac{1}{2}-\frac{\mu K_{GN}}{p}c_n^{\frac{6-p}{4}}
 \rho_0^{\frac{3(p-2)}{4}-1}-\frac{1}{6\mathcal{S}^3}\rho_0^2\Big)
>0,
\end{align*}
which yields a contradiction.
\end{proof}

\begin{lemma}\label{strong subadd condi}
 Let $c\in(0,c_0)$. Then the following strict subadditivity inequality holds,
 \[
 m(c)<m(c_1)+m(c_2),
 \]
 where $c=c_1+c_2$ and $0<c_1,c_2<c$.
\end{lemma}

\begin{proof}
 Note that from Proposition \ref{strong subadd}, we just need to show the
condition \eqref{md condi4} holds. For $u\in M(c)$, without loss of generality,
we suppose that there exists some $\tilde{c}\in(0,c]$ such that
 $\|u\|_2^2=\tilde{c}$ and $I(u)=m(\tilde{c})$.
 Since $u$ is a minimizer of $I(u)$ on $V(\tilde{c})$, we deduce from
Lemma \ref{npidentity} that
 \begin{equation}\label{strong add condition eq1}
A(u)+\frac{1}{4}B(u)-\frac{1}{16\pi}\int_{\mathbb{R}^3}
\int_{\mathbb{R}^3}e^{-|x-y|}u^2(x)u^2(y)\,dx\,dy-\frac{3\mu(p-2)}{2p}C(u)-D(u)=0.
 \end{equation}

For $u\neq0$ we compute $h_{g_u}(\theta)$ by considering the family of scaling
paths of $u$ parameterized with $\beta\in\mathbb{R}$ given by
$u_{\theta}(x):=\theta^{\frac{1+3\beta}{2}}u(\theta^\beta x)$.
By a straightforward computation, we have
\begin{gather*}
A(u_{\theta})=\theta^{1+2\beta}A(u),\quad
B(u_{\theta})=\theta^{2+\beta}H(u)-\frac{\theta^{2+\beta}}{4\pi}
 \int_{\mathbb{R}^3}\int_{\mathbb{R}^3}
 \frac{e^{-\frac{|x-y|}{\theta^\beta}}}{|x-y|}u^2(x)u^2(y)\,dx\,dy,\\
C(u_{\theta})=\theta^{\frac{(1+3\beta)p}{2}-3\beta}C(u),\quad
D(u_{\theta})=\theta^{3(1+2\beta)}D(u),~~\|u_{\theta}\|_2^2=\theta\|u\|_2^2,
 \end{gather*}
where $H(u)=\frac{1}{4\pi}\int_{\mathbb{R}^3}\int_{\mathbb{R}^3}
\frac{u^2(x)u^2(y)}{|x-y|}\,dx\,dy$.

Let $h_{g_u}(\theta)=f(\theta,u):=I(u_{\theta})-\theta I(u)$. Then
 \begin{equation}\label{strict add fun1}
\begin{aligned}
 h_{g_u}(\theta)
&=f(\theta,u)\\
&=\frac{1}{2}(\theta^{1+2\beta}-\theta)A(u)
 +\frac{1}{4}\Big[\theta^{2+\beta}H(u) \\
&\quad -\frac{\theta^{2+\beta}}{4\pi}
 \int_{\mathbb{R}^3}\int_{\mathbb{R}^3}
 \frac{e^{-\frac{|x-y|}{\theta^\beta}}}{|x-y|}u^2(x)u^2(y)\,dx\,dy
 -\theta B(u)\Big]\\
&\quad -\frac{\mu}{p}(\theta^{\frac{(1+3\beta)p}{2}-3\beta}-\theta)C(u)
 -\frac16(\theta^{3(1+2\beta)}-\theta)D(u).
 \end{aligned}
\end{equation}
Moreover, we can deduce that
 \begin{align*}
 f'_{\theta}(\theta,u)
&=\frac{1}{2}\left((1+2\beta)\theta^{2\beta}-1\right)A(u)
 +\frac14\Big[(2+\beta)\theta^{1+\beta}H(u)\\
&\quad -\frac{(2+\beta)\theta^{1+\beta}}{4\pi}\int_{\mathbb{R}^3}\int_{\mathbb{R}^3}
 \frac{e^{-\frac{|x-y|}{\theta^\beta}}}{|x-y|}u^2(x)u^2(y)\,dx\,dy\\
&\quad -\frac{\theta^{2+\beta}}{4\pi}\frac{\beta}{\theta^{\beta+1}}
 \int_{\mathbb{R}^3}\int_{\mathbb{R}^3}e^{-\frac{|x-y|}{\theta^\beta}}u^2(x)u^2(y)
 \,dx\,dy-B(u)\Big]\\
&\quad -\frac{\mu}{p}\Big((\frac{(1+3\beta)p}{2}-3\beta)\theta^{\frac{(1+3\beta)p}{2}
 -3\beta-1}-1\Big)C(u)\\
&\quad -\frac16\left((3(1+2\beta))\theta^{3(1+2\beta)-1}-1\right)D(u),
 \end{align*}
which leads to
 \begin{equation}\label{strict add fun2}
\begin{aligned}
&f'_{\theta}(1,u) \\
&=\beta A(u)+\frac{1}{4}\Big[-\frac{\beta}{4\pi}\int_{\mathbb{R}^3}
 \int_{\mathbb{R}^3}e^{-|x-y|}u^2(x)u^2(y)\,dx\,dy+(1+\beta)B(u)\Big]\\
&\quad -\frac{\mu}{p}(\frac{(1+3\beta)p}{2}-3\beta-1)C(u)-\frac16(3(1+2\beta)-1)D(u).
 \end{aligned}
\end{equation}
Now it remains to show that the admissible scaling path satisfies
$h'_{g_u}(1)\neq0$. Again, we process by way of contradiction: assume that there
exists a sequence $\{u_n\}\subset M(c)$ with $\|u_n\|_2^2=c_n\leq c$ and $c_n\to0$
such that for all $\beta\in\mathbb{R}$, we have $f'_{\theta}(1,u_n)=0$. That is,
 \begin{equation}\label{strong add condition eq2}
\begin{aligned}
&\beta A(u_n)+\frac{1}{4}\Big[-\frac{\beta}{4\pi}\int_{\mathbb{R}^3}
 \int_{\mathbb{R}^3}e^{-|x-y|}u_n^2(x)u_n^2(y)\,dx\,dy+(1+\beta)B(u_n)\Big]\\
&-\frac{\mu}{p}(\frac{(1+3\beta)p}{2}-3\beta-1)C(u_n)-\frac16(3(1+2\beta)-1)D(u_n)
=0.
 \end{aligned}
\end{equation}
Combining \eqref{strong add condition eq1} and \eqref{strong add condition eq2}
yields
 \begin{gather}\label{strong add condition eq3}
 \frac{1}{4}B(u_n)=\frac{\mu(p-2)}{2p}C(u_n)+\frac{1}{3}D(u_n), \\
\label{strong add condition eq4}
 B(u_n)=2A(u_n)-\frac{1}{8\pi}\int_{\mathbb{R}^3}\int_{\mathbb{R}^3}
e^{-|x-y|}u_n^2(x)u_n^2(y)\,dx\,dy.
 \end{gather}
 Moreover, from the continuity of $m(c)$ and the Gagliardo-Nirenberg inequality
\eqref{gn}, we have
 \begin{equation}\label{strong add condition eq5}
 \begin{gathered}
  I(u_n)=m(c_n)\to 0,\\
  A(u_n), B(u_n), C(u_n), D(u_n)\to 0.
 \end{gathered}
 \end{equation}
We need to consider three cases.
\smallskip

\noindent\textbf{Case 1:} $2<p<12/5$.
From the Hardy-Littlewood-Sobolev inequality, the interpolation inequality,
the Sobolev embedding theorem and \eqref{strong add condition eq3}, it follows that
 \begin{align*}
 B(u_n)
&=\frac{1}{4\pi}\int_{\mathbb{R}^3}\int_{\mathbb{R}^3}\frac{1-e^{-|x-y|}}{|x-y|}u_n^2(x)u_n^2(y)\,dx\,dy\\
 &\leq\frac{1}{4\pi}\int_{\mathbb{R}^3}\int_{\mathbb{R}^3}\frac{u_n^2(x)u_n^2(y)}{|x-y|}\,dx\,dy\\
 &\leq C\|u_n\|_{12/5}^4\\
 &\leq C\|u_n\|_{p}^{\frac{6p}{6-p}}\|u_n\|_{6}^{\frac{12-5p}{3p}}\\
 &=CC(u_n)^{\frac{6}{6-p}}A(u_n)^{\frac{12-5p}{6p}}\\
 &\leq C_1B(u_n)^{\frac{6}{6-p}}A(u_n)^{\frac{12-5p}{6p}}.
\end{align*}
This leads to
 \[
 1\leq C_1B(u_n)^{\frac{p}{6-p}}A(u_n)^{\frac{12-5p}{6p}},
 \]
which yields a contradiction with \eqref{strong add condition eq5}.
\smallskip

\noindent\textbf{Case 2:} $p=\frac{12}{5}$. Due to \eqref{strong add condition eq3},
we obtain
 \[
 \|u_n\|_{12/5}^{12/5}\leq\frac{3}{\mu}B(u_n)\leq C_2\|u_n\|_{12/5}^4,
 \]
 which is impossible because of the fact $\|u_n\|_{12/5}\to0$.
\smallskip

\noindent\textbf{Case 3:} $12/5<p<8/3$.
From \eqref{strong add condition eq3} it follows that
 \[
 \|u_n\|_p^p\leq\frac{p}{2\mu(p-2)}B(u_n)\leq C_3\|u_n\|_{12/5}^4\leq C_3\|u_n\|_2^{\frac{2(5p-12)}{3(p-2)}}\|u_n\|_p^{\frac{2p}{3(p-2)}}.
 \]
This leads to
 \[
 1\leq C_3 c_n^{{\frac{5p-12}{3(p-2)}}}\|u_n\|_p^{\frac{p(8-3p)}{3(p-2)}},
 \]
 which yields another contradiction with \eqref{strong add condition eq5}.

All the hypotheses of Proposition \ref{strong subadd} have been
verified and thus we have arrived at the desired result.
\end{proof}

\begin{proof}[Proof of Theorem \ref{thm: main result}]
 For any $c\in(0,c_0)$, we assume that $\{u_n\}\subset V(c)$ satisfies
$I(u_n)\to m(c)$. By Lemma \ref{nonvanish pro}, there exists a sequence
$\{y_n\}\subset\mathbb{R}^3$ such that
 \[
 u_n(x-y_n)\rightharpoonup u\neq0~~\text{in}~~H^1(\mathbb{R}^3).
 \]
We start by showing that $w_n(x):=u_n(x-y_n)-u(x)\to0$ in $H^1(\mathbb{R}^3)$.
Clearly, we can see that
 \begin{gather*}
\|u_n\|_2^2=\|w_n\|_2^2+\|u\|_2^2+o_n(1),\\
\|\nabla u_n\|_2^2=\|\nabla w_n\|_2^2+\|\nabla u\|_2^2+o_n(1),\\
I(u_n)=I(w_n)+I(u)+o_n(1).
 \end{gather*}
The last equality holds because of the translational invariance.
Then we claim that $\|w_n\|_2^2\to0$. Let $\|u\|_2^2=c_1>0$. It suffices to
show that $c_1=c$. We assume by contradiction that $c_1<c$. Since we have,
for $n$ large enough, $\|w_n\|_2^2\leq c$ and
$\|\nabla w_n\|_2^2\leq\|\nabla u_n\|_2^2<\rho_0$.
Then $w_n\in V(\|w_n\|_2^2)$ and $I(w_n)\ge m(\|w_n\|_2^2)$, which implies that
 \[
 m(c)=I(w_n)+I(u)+o_n(1)\ge m(\|w_n\|_2^2)+I(u)+o_n(1).
 \]

By Lemma \ref{mcproperty} (ii), we have
 \[
 m(c)\ge m(c-c_1)+I(u).
 \]
Moreover, we see that $u\in V(c_1)$, then $I(u)\ge m(c_1)$,
and from Lemma \ref{strong subadd condi} it follows that
 \[
 m(c)\ge m(c-c_1)+m(c_1)>m(c),
 \]
which is impossible. Hence the claim follows and $\|u\|_2^2=c$.

Now we show that $\|\nabla w_n\|_2^2\to0$. Since $u\neq0$, we have
$\|\nabla w_n\|_2^2\leq\|\nabla u_n\|_2^2<\rho_0$ for $n$ large enough.
Thus $\{w_n\}\subset B_{\rho_0}$ and $\{w_n\}$ is bounded in $H^1(\mathbb{R}^3)$.
By Lemma \ref{CDestimate} (ii), recalling $\|w_n\|_2^2\to0$, we obtain
$\|w_n\|_p^p\to0$. Then, from Lemma \ref{nonvanish pro1} we can deduce that
 \begin{equation}\label{proof of main re eq1}
 I(w_n)\ge\beta_0\|\nabla w_n\|_2^2+o_n(1)~~\text{for}~~\beta_0>0.
 \end{equation}
Since $u\in V(c)$, we obtain $I(u)\ge m(c)$. Then
 \[
 I(u_n)=I(u)+I(w_n)+o_n(1)\rightarrow m(c),
 \]
 which implies
 \[
 I(w_n)\leq o_n(1).
 \]
 Taking into account \eqref{proof of main re eq1}, we can see that
$\|\nabla w_n\|_2^2\to0$. Therefore, $u_n\to u$ holds in $H^1(\mathbb{R}^3)$.
Moreover, $u$ is a minimizer for $I$ on $V(c)$.

 By  Lemma \ref{npidentity}, we know that all the minimizers of the functional
$I$ restricted on $S(c)$ lie in $\Lambda(c)$. We define
 \[
 \bar{m}(c):=\inf_{u\in\Lambda(c)}I(u).
 \]
 By Lemma \ref{npidentity coer}, $I$ restricted to $\Lambda(c)$ is bounded
from below, which means that $\bar{m}(c)$ is well-defined.
By an analogous argument, we can see that Proposition \ref{strong subadd}
also holds for $\bar{m}(c)$. This indicates that any minimizing sequence
$\{\bar{u}_n\}$ on $\Lambda(c)$ is relatively compact, i.e.,
$\bar{u}_n\to \bar{u}$ in $H^1(\mathbb{R}^3)$. It is easily seen that
$\{\bar{u}_n\}$ is a bounded Palais-Smale sequence for $I$ on $S(c)$, and thus
$\bar{u}$ is a ground state solution of \eqref{SBP} on $S(c)$.
\end{proof}

\subsection*{Acknowledgments}
This work is supported by National Natural Science Foundation of China No.\
11971095.

\end{document}